\documentclass[12pt,reqno]{amsart}
\usepackage{amscd,amssymb,epsfig}
\calclayout

\numberwithin{equation}{section} 

\newcommand{\ud}{\,d} 
 
\newcommand{\R}{\mathbb{R}}

\newcommand{\tir}[1]{\ensuremath{\overline {#1}}} 
\newtheorem{thm}{Theorem}[section] 
 
\newtheorem{lemma}[thm]{Lemma} 
 
\newtheorem{prop}[thm]{Proposition} 
 
\newtheorem{defn}[thm]{Definition}

\def\whsq{\vbox to 5.8pt 
{\offinterlineskip\hrule 
\hbox to 5.8pt{\vrule height 
5.1pt\hss\vrule height 5.1pt}\hrule}}

\def\<{\langle} 
\def\>{\rangle} 
\def\PP{{\mathop{{\rm I}\kern-.2em{\rm P}}\nolimits}} 
\def\FF{{\mathop{{\rm I}\kern-.2em{\rm F}}\nolimits}}   
\def\ZZ{{\mathop{{\rm I}\kern-.2em{\rm Z}}\nolimits}} 
 
\newlength{\sidemargin} 
\begin{document}
\title[]{
Convergence of 
finite difference schemes to the Aleksandrov solution of the Monge-Amp\`ere equation
}

\thanks{Gerard Awanou was partially supported by NSF grant DMS-1319640. }
\author{Gerard Awanou}
\address{Department of Mathematics, Statistics, and Computer Science, M/C 249.
University of Illinois at Chicago, 
Chicago, IL 60607-7045, USA}
\email{awanou@uic.edu}  
\urladdr{http://www.math.uic.edu/\~{}awanou}

\author{Romeo Awi}
\address{School of Mathematics, Georgia Institute of Technology, 686 Cherry Street
Atlanta, GA 30332-0160, USA}
\email{rawi3@math.gatech.edu}
\urladdr{http://www.math.gatech.edu/~rawi3}

\maketitle

\begin{abstract}
We present a technique for proving convergence to the Aleksandrov solution of the Monge-Amp\`ere equation of a stable and consistent finite difference scheme. 
We also require a notion of discrete convexity with a stability property and a local equicontinuity property for bounded sequences.

\end{abstract}

\section{Introduction}

Given an orthogonal lattice with mesh length $h$ on a convex bounded domain $\Omega \subset \R^d$ with boundary $\partial \Omega$, we are interested in convergent finite difference approximations of the problem: find a convex function
$u \in C(\tir{\Omega})$ such that
\begin{align} \label{m1}
\begin{split}
\det D^2 u &=\nu \, \text{in} \, \Omega\\
 u &=g \, \text{on} \, \partial \Omega,
 \end{split}
\end{align}
where $\nu$ is a finite Borel measure 
and $g \in C(\partial \Omega)$ can be extended to a  convex function $\tilde{g} \in C(\tir{\Omega})$. When $u \in C^2(\Omega)$, $\det D^2 u$ is the determinant of the Hessian matrix $D^2 u=\bigg( \frac{\partial^2 u}{\partial x_i \partial x_j}\bigg)_{i,j=1,\ldots, d}$. In the general case, the expression $\det D^2 u$ denotes the Monge-Amp\`ere measure associated with $u$.

Let $\Omega_h$ denote the computational domain and $\partial \Omega_h$ its boundary. Let $f_h \geq 0$ be a family of mesh functions which converge to $\nu$ as measures. We consider the problem 
with unknown a mesh convex function $u_h$
\begin{align} \label{m1h}
\begin{split}
h^d \mathcal{M}_h[u_h] &= h^d f_h \, \text{in} \, \Omega_h\\
 u_h &=g \, \text{on} \, \partial \Omega_h.
 \end{split}
\end{align}
Here $\mathcal{M}_h[v_h]$ denotes a 
stable and consistent discretization of $\det D^2 v$ for a smooth convex function $v$. 
There are several notions of discrete convexity. We require that the uniform limit on compact subsets of mesh convex functions is a convex function and that a locally bounded sequence of such functions is locally equicontinuous.  
Of course we also require \eqref{m1h} to have a solution. A sufficient condition
is degenerate ellipticity and Lipschitz continuity as defined by Oberman \cite{Oberman06}.
We show that a family of solutions $u_h$ of \eqref{m1h} converges uniformly on compact subsets to the unique Aleksandrov solution of \eqref{m1}.

The Monge-Amp\`ere equation \eqref{m1} is a fully nonlinear equation which arises in several applications of great importance, e.g. optimal transportation and reflector design. Problems in affine geometry motivated the study of the Dirichlet 
problem. 

The equation $\det D^2 u = \nu$ with 
$\nu$ a sum of Dirac masses, and with Dirichlet boundary condition was solved by Pogorelov \cite{Pogorelov73}. For the so-called second boundary condition we refer to
\cite[Chapter V section 3]{Pogorelov64}.  When the measure $\nu$ is absolutely continuous with respect to the Lebesgue measure, with density $f \geq 0$ and $f \in C(\Omega)$, the convergence of a scheme of the type \eqref{m1h} was proved in \cite{Oberman2010a} using the notion of viscosity solution. In this paper, we use the notion of Aleksandrov solution, the consistency of the discretization \eqref{m1h} and approximation by smooth functions to handle the Monge-Amp\`ere measure.  In \cite{Awanou-Std-fd} the notion of Aleksandrov solution was also used along with a different procedure for approximation by smooth functions.

We note that the method introduced in  \cite{Benamou2014b} which is more effective when the measure $\nu$ is a combination of Dirac masses is not consistent. Our requirements for convergence which are stability, consistency and solvability of \eqref{m1}, as well as stability under uniform convergence on compact subsets of discrete convex mesh functions along with local equicontinuity of locally bounded sequences of such functions, 
are met by the finite difference scheme introduced in \cite{Oberman2010a}. However our
 numerical results indicate that a very good initial guess is required for an iterative method for solving the nonlinear problem \eqref{m1h} if one uses the discretization proposed in \cite{Oberman2010a}. In our numerical
experiments, the discrete problem \eqref{m1h} is solved with a time marching method which has also been used in \cite{AwanouHybrid}. The difficulty of capturing singular solutions may be related to the choice of the method for solving the nonlinear equation \eqref{m1h}. 
We plan to use in a subsequent work semi-deterministic algorithms of the type introduced in \cite{Mohammadi2007} for hyperbolic  Monge-Amp\`ere equations. This is motivated by the observation that discretizations of the type introduced in \cite{Oberman2010a} may have multiple solutions.

If $\nu$ has density $f \geq 0$ and $f \in C(\tir{\Omega})$, under our assumptions, $u_h$ converges uniformly on compact subsets to the unique viscosity solution of \eqref{m1} if the latter is known to have a unique viscosity solution. This follows from the 
equivalence of the notion of viscosity and Aleksandrov solutions, the proof of which we outline. If $f >0 $ and $f \in C(\tir{\Omega})$, a continuous viscosity solution of \eqref{m1} is also an Aleksandrov solution of \eqref{m1}  \cite[Proposition 1.7.1]{Guti'errez2001}. The result is also valid for $f \geq 0$ and $f \in C(\tir{\Omega})$. Indeed, consider the problems $\det D^2 u_{\epsilon} = f+ \epsilon, \epsilon >0$. By \cite[Lemma 5.1 ]{Hartenstine2006},  $u_{\epsilon}$ converges uniformly on compact subsets to $u$. One then uses the equivalence of the notion of viscosity and Aleksandrov solutions in the non degenerate case \cite[Propositions 1.7.1 and 1.3.4]{Guti'errez2001} and the stability of the notion of Aleksandrov and viscosity solutions under uniform convergence on compact subsets \cite[Lemma 1.2.3]{Guti'errez2001} and \cite[Theorem 2.3]{Barles94}. 

In this paper we provide the convergence proof of a time marching method for solving the nonlinear problem \eqref{m1h}.  The main contribution of this paper is the method of proof for convergence of finite difference schemes satisfying our assumptions. The numerical results indicate that such schemes may not lead to an effective numerical algorithm. Our results clarify the nature of efficient discretizations for \eqref{m1}. 
Another consequence of our results is the equivalence of the notions of viscosity and Aleksandrov solutions for $f \geq 0$ and $f \in C(\Omega) \cap L^1(\Omega)$. Indeed as we show in section \ref{numerical}, $u_h$ obtained through \eqref{m1h} and the discretization proposed in \cite{Oberman2010a} converges to the Aleksandrov solution. It is also known to converge to the unique viscosity solution of \eqref{m1} when the latter exists. Hence the result.

The paper is organized as follows: in the next section we give some notation, recall key results on the Aleksandrov solution and finite difference schemes. In section \ref{conv} we prove the claimed convergence result. 
We conclude with a proof of convergence of the time marching method and numerical experiments.

\section{Preliminaries}

In this section, we recall key results on the Aleksandrov solution of the Monge-Amp\`ere equation. We then associate discrete measures to mesh functions. For a smooth solution of \eqref{m1} we immediately get a discretization of the Monge-Amp\`ere measure. 
Finally, we introduce finite difference schemes.

\subsection{The Monge-Amp\`ere measure}

In this paper, we take the analytic approach to the Monge-Amp\`ere measure \cite{Rauch77}. Let $K(\Omega)$ denote the cone of convex functions on $\Omega$ and let $M(\Omega)$ denote the set of Borel measures on $\Omega$. For $v \in C^2(\Omega) \cap K(\Omega)$, and given a Borel set $B$, we define a Borel measure $\mathcal{M}[v]$ by
$$
\mathcal{M}[v](B) = \int_B \det D^2 v(x) \ud x.
$$

The topology on $M(\Omega)$ is induced by the weak convergence of measures

\begin{defn}
A sequence $\mu_n \in M(\Omega)$ converges to $\mu \in M(\Omega)$ if and only if $\mu_n(B) \to \mu(B)$ for any Borel set $B \subset \Omega$ with $\mu(\partial B)=0$.
\end{defn}
We note that there are several equivalent definitions of weak convergence of measures which can be found for example in \cite[Theorem 1, section 1.9]{Evans-Gariepy}. We have

\begin{prop} \cite[Proposition 3.1]{Rauch77} \label{weak-c}
The mapping $\mathcal{M}$ maps $C(\Omega)$-bounded subsets of $C^2(\Omega) \cap K(\Omega)$ into bounded subsets of $M(\Omega)$. Moroever $\mathcal{M}$ has a unique extension to a continuous operator on $K(\Omega)$.
\end{prop}

For a convex function $v$, we will refer to $\mathcal{M}[v]$ as the Monge-Amp\`ere measure associated with $v$. It can be shown that it coincides with the notion of Monge-Amp\`ere measure obtained through the normal mapping \cite[Proposition 3.4]{Rauch77}.

\begin{defn}
Given a Borel measure $\nu$ on $\Omega$, a convex function $u \in C(\Omega)$ is an Aleksandrov solution of
$$
\det D^2 u = \nu,
$$
if and only if $\mathcal{M}[u] = \nu$.
\end{defn}

We recall an existence and uniqueness result for the solution of \eqref{m1}.

\begin{prop}  [\cite{Hartenstine2006} Theorem 1.1] \label{ex-Alex}
Let $\Omega$ be a bounded convex domain of $\R^d$. Assume $\nu$ is a finite Borel measure and $g \in C(\partial \Omega)$ can be extended to a function $\tilde{g} \in C(\tir{\Omega})$ which is  convex in $\Omega$. Then the Monge-Amp\`ere equation \eqref{m1}
has a unique Aleksandrov solution in $K(\Omega) \cap C(\tir{\Omega})$.

\end{prop}

Throughout this paper, we will follow the convention of denoting by $p$  a measure $\nu$ which is absolutely continuous with respect to the Lebesgue measure and with density $p$.

\subsection{Discrete measures associated with mesh functions}

Let $h$ be a small positive parameter and let
$$
\mathbb{Z}^d_h = \{\, m h, m \in \mathbb{Z}^d \, \},
$$
denote the regular uniform grid of $\R^d$. By a mesh function we mean a real-valued function defined on $\mathbb{Z}^d_h$. We denote by $\mathcal{C}_h$ the cone of discrete convex mesh functions. We consider in 
section \ref{numerical} a notion of discrete convex mesh functions which fulfill the assumptions of this paper.

The computational domain is defined as $\Omega_h = \Omega \cap \mathbb{Z}^d_h$ and its boundary is simply $\partial \Omega_h = \{ \, x \in \tir{\Omega} \cap  \mathbb{Z}_h^d,  x \notin \Omega_h \, \}$.


Let $v_h$ be a mesh function such that $v_h \geq 0$ on $\Omega_h$. We associate to $v_h$ a Borel measure which we denote here by $v_h$ and defined by
\begin{equation} \label{defn-dbm}
v_h(B) = h^d \sum_{x \in \Omega_h \cap B} v_h(x),
\end{equation}
for any Borel set $B$. 

Given a continuous function $v$ on  $\Omega$, we use the notation $r_h(v)$ to denote its restriction to $\Omega_h$. If $v \geq 0$ is a continuous function on $\Omega$, we have 
$$
\lim_{h \to 0} r_h(v)(B) = \int_B v(x) \ud x.
$$
for any $B\subset \Omega$ satisfiying 
$|\partial B|=0$.
In other words, the measures $v_h$ converge weakly to $v$. 

\begin{defn} \label{disc-uc}
Let $v_h \in \mathcal{C}_h$ for each $h >0$. We say that $v_h$ converges to a convex function $v \in C(\Omega)$ uniformly on compact subsets of $\Omega$ if and only if for each compact set $K \subset \Omega$, each sequence $h_k \to 0$ and for all $\epsilon >0$, there exists $h_{-1} >0$ such that for all $h_k$, $0< h_k < h_{-1}$, we have
$$
\max_{x \in K \cap \mathbb{Z}_h^d} |v_{h_k}(x) - v(x)| < \epsilon.
$$
\end{defn}

\subsection{Finite difference schemes}
Let $f_h \geq 0$ be a family of mesh functions which converge to $\nu$ as measures and let $\mathcal{M}_h[r_h v ]$ denotes a  discretization of $\det D^2 v$ for a smooth convex function $v$. We are interested in the discrete Monge-Amp\`ere equation \eqref{m1h}. 

The discretization $\mathcal{M}_h[r_h v]$ is said to be consistent if for all $C^2$ convex functions $v$, and a sequence $x_h \in \Omega_h$ such that $x_h \to x \in \Omega$,
$\lim_{h \to 0} \mathcal{M}_h[r_h v] (x_h) = \det D^2 v(x)$.  

The discretization is said to be stable if the problem $ \mathcal{M}_h[v_h]  = f_h$ has a solution $v_h$ which is bounded independently of $h$.



\section{Convergence} \label{conv}
We recall that our assumptions are that the nonlinear equation \eqref{m1h} is solvable, the discretization is stable and consistent and finally the uniform limit on compact subsets of mesh convex functions is a convex function and that a locally bounded sequence of such functions is locally equicontinuous. 

\begin{lemma} \label{conj2}
There exists a constant $C_0$ such that for all Borel sets $B \subset \Omega$
$$
| \mathcal{M}_h[w_{h}](B) - \mathcal{M}_h[v_h](B) | \leq C_0 \max_{x \in \tir{B} \cap \mathbb{Z}^d_h}| w_h(x)-v_h(x)|. 
$$
\end{lemma}

\begin{proof}
We have by \eqref{defn-dbm}

\begin{align*}
| \mathcal{M}_h[w_{h}](B) - \mathcal{M}_h[v_h](B) | & = h^d \bigg| \sum_{x \in B \cap \Omega_h} w_h(x) -  \sum_{x \in B \cap \Omega_h}  v_h(x) \bigg| \\
& \leq \sum_{x \in \tir{B} \cap \Omega_h} h^d | w_h(x)-v_h(x)| \\
& \leq \max_{x \in \tir{B} \cap \mathbb{Z}^d_h}| w_h(x)-v_h(x)|  \sum_{x \in \tir{B} \cap \Omega_h} h^d \\
& \leq C  \max_{x \in \tir{B} \cap \mathbb{Z}^d_h}| w_h(x)-v_h(x)|. 
\end{align*}
\end{proof}

The constant $C_0$ in Lemma \ref{conj2} satisfies
\begin{equation} \label{c0}
 \sum_{x \in \Omega_h} h^d \leq C_0.
\end{equation}

\begin{lemma} \label{conj1}
For $v \in K(\Omega) \cap C^2(\Omega)$,  $h^d \mathcal{M}_h[r_h v]$ converge weakly to $\det D^2 v$.
\end{lemma}

\begin{proof}
By consistency, $\mathcal{M}_h[r_h v](x) \to \det D^2 v(x)$ for $v \in K(\Omega) \cap C^2(\Omega)$. By definition of the Monge-Amp\`ere measure and of discretization of the integral, and using consistency, we have
for a Borel set $B$ with $\mathcal{M}[v](\partial B)=0$
\begin{align*}
\int_B \det D^2 v(x) \ud x & = \lim_{h \to 0} h^d \sum_{x \in B \cap \Omega_h} \det D^2 v(x).
\end{align*}
Moreover using \eqref{c0}
\begin{align*}
\bigg|  h^d \sum_{x \in B \cap \Omega_h} \det D^2 v(x) - h^d \mathcal{M}_h[r_h v](B) \bigg| & = \bigg|   \sum_{x \in B \cap \Omega_h} h^d\left(\det D^2 v(x)  -  \mathcal{M}_h[r_h v](x)\right) \bigg|  \\
& \leq C_0 \max_{x \in B \cap \Omega_h} |  \det D^2 v(x)  -  \mathcal{M}_h[r_h v](x)|,
\end{align*}
from which the result follows by consistency.
\end{proof}

We have the following weak convergence result for discrete Monge-Amp\`ere measures

\begin{thm}
Let $v_h \in \mathcal{C}_h$ converge uniformly on compact subsets to $v \in K(\Omega) \cap C(\tir{\Omega})$. Then $h^d \mathcal{M}_h[v_h]$ converges weakly to $\mathcal{M}[v]$.
\end{thm}

\begin{proof}
Let $v_{\epsilon} \in K(\Omega) \cap C^2(\Omega)$ converge uniformly to $v$ on $\Omega$. The existence of $v_{\epsilon}$ may be proven as in \cite{MongeC1Alex}. Let $B$ be a Borel set with $\mathcal{M}[v](\partial B)=0$. Given $\delta >0$, we seek $h_0 >0$ such that
$| h^d \mathcal{M}_h[v_h](B) - \mathcal{M}[v](B) | < \delta $ for all $0<h<h_0$.

By Proposition \ref{weak-c}, $\exists \, \epsilon_0> 0$ such that $| \mathcal{M}[v_{\epsilon_0}](B) - \mathcal{M}[v](B) | < \delta/3$. By Lemma \ref{conj1}, $\exists h_0 >0$ such that for all $0<h<h_0$, $| h^d \mathcal{M}_h[(r_h(v_{\epsilon_0})](B) - \mathcal{M}[v_{\epsilon_0}](B) | < \delta/3$.

We may assume that $ \max_{x \in \tir{B}} |v_{\epsilon_0}(x) - v(x)| < \delta /(6 C_0)$ and 
since $v_h$ converges to $v$ on $\tir{B}$, we may assume that for $h<h_0$, $ \max_{x \in \tir{B} \cap \mathbb{Z}^d_h} |v(x) - v_h(x) | < \delta/(6 C_0)$. Thus we have
$ \max_{x \in \tir{B} \cap \mathbb{Z}^d_h}  |v_{\epsilon_0}(x) - v_h(x) | < \delta/(3 C_0)$. By Lemma \ref{conj2}, $| \mathcal{M}_h[r_h(v_{\epsilon_0})](B) - \mathcal{M}_h[v_h](B) | < \delta/3$.
This concludes the proof.

\end{proof}

We can now prove the main result of this paper

\begin{thm}
The mesh function $u_h$ defined by \eqref{m1h} converges uniformly on compact subsets to the Aleksandrov solution $u$ of \eqref{m1}.
\end{thm}

\begin{proof}
By the stability assumption, the family $u_h$ is uniformly bounded and by our assumption on the discretization, locally equicontinuous. By the Arzela-Ascoli theorem, there exists a subsequence $u_{h_k}$ which converges uniformly on compact subsets to a function 
$v$. Since $u_h \in \mathcal{C}_h$ the function $v$ is convex by our assumptions on discrete convex functions. 
Since $u_h$ is uniformly bounded, $v$ is convex and bounded on $\Omega$, hence continuous on $\Omega$. 
Arguing as in the proof of \cite[Theorem 4.3]{DiscreteAlex} one proves that $v \in C(\tir{\Omega})$.
Since $h^d\mathcal{M}_h[v_h]$ converge weakly to $\mathcal{M}[v]$ and $u_h=g$ on $\partial \Omega$, the function $v$ is an Aleksandrov solution of $\eqref{m1}$. By uniqueness, $v=u$ and hence the whole family $u_h$ converges uniformly on compact subsets to $u$. 
\end{proof}


\section{Convergence of a time marching iterative method } 
Let us denote by $\mathcal{M}(\Omega^h)$ the set of mesh functions, i.e. the set of real valued functions defined on $\Omega^h$. Since $\Omega_h$ is a finite set, there is a canonical identification of $\mathcal{M}(\Omega^h)$ with $\R^N$ for some integer $N$. We will now also use the restriction operator $r_h$ for
 vector and matrix fields.  For $x \in \R^N$,  $|x|=(\sum_{i=1}^N x_i^2)^{\frac{1}{2}}$ denotes the Euclidean norm of $x$  and  $|x|_{\infty}=\max_{i=1,\ldots,N} |x_i|$  denotes its maximum norm.
 
We make the assumption that the mapping $\mathcal{M}_h$ is Lipschitz continuous with Lipschitz constant $K >0$ i.e.
\begin{equation*}
|\mathcal{M}_h[v_h] - \mathcal{M}_h[w_h] |_{\infty} \leq K |v_h-w_h|_{\infty}, v_h, w_h \in \R^N.
\end{equation*}
Here we make the abuse of notation of identifying a mesh function with its vector representation.
We also make the assumption that problem \eqref{m1h} has a unique solution $u_h$ which can be computed by a time marching method
\begin{align} \label{m1h-iter}
\begin{split}
u_h^{k+1} & = u_h^{k} + \frac{1}{\mu} \mathcal{M}_h[u_h^k] \, \text{in} \,  \Omega_h \\
u_h & = r_h(g) \, \text{on} \, \partial \Omega_h,
\end{split}
\end{align}
for $\mu \geq \mu_0$ and $u_h^0$ a suitable initial guess. Such assumptions are satisfied by proper Lipschitz continuous degenerate elliptic schemes as defined by Oberman \cite{Oberman06}.
Although our theory indicates convergence of the discretization for the case where the measure $\mu$ is a combination of Dirac masses, we were not able to get numerical evidence of convergence for the above iterative 
method for the discretization proposed in \cite{Oberman2010a} even if we use the exact solution as initial guess. Similar results for Newton's method were reported in \cite{Benamou2014b}.

Let us denote by $\Delta_h$ the standard finite difference discretization of the Laplace operator and let  $e_i$ denote the ${i}^{th}$ vector of the canonical basis of $\mathbb R^d$. For $x \in \Omega_h$ and $v_h \in \mathcal{M}(\Omega^h)$, we have
\begin{equation*}
\Delta_h v_h (x)  = \sum_{i=1}^d \frac{v_h(x+h e_i) - 2 v_h(x) + v_h(x-h e_i)}{h^2}.
\end{equation*}
When the measure $\mu$ is a combination of Dirac masses we obtained better numerical results with the preconditioned iterative method
\begin{align} \label{m1h-iter2}
\begin{split}
-\Delta_h u_h^{k+1} & = -\Delta_h u_h^{k} + \frac{1}{\mu} \mathcal{M}_h[u_h^k] \, \text{in} \,  \Omega_h \\
u_h & = r_h(g) \, \text{on} \, \partial \Omega_h,
\end{split}
\end{align}
for $\mu \geq \mu_1$ under the above assumptions. Moreover numerical experiments indicate that the method \eqref{m1h-iter2} converges faster than \eqref{m1h-iter}. The idea to use the Laplacian for faster iterative methods has a long story in various contexts \cite{Farago02} p. 58, and a remark in that direction for proper Lipschitz continuous degenerate elliptic schemes was made in \cite{GlowinskiICIAM07}. See also \cite{Oberman13}. We use the terminology preconditioned iterative method for \eqref{m1h-iter2} by analogy with preconditioned techniques for linear equations. An advantage of the preconditioned iterative method  \eqref{m1h-iter2} is that fast Poisson solvers and standard multigrid methods can be used at each step.

The proof of convergence of the iterative method \eqref{m1h-iter2} does not follow the approach in \cite{Oberman06} for proving convergence of the basic iterative method  \eqref{m1h-iter}. The proof of the latter does not seem to extend to the preconditioned version \eqref{m1h-iter2}. We take a different approach which consists in using the fact that  \eqref{m1h-iter} converges to the discrete solution of \eqref{m1h} and properties of the inverse of the operator $ \Delta_h$.

\subsection{Convergence of the preconditioned iterative method}

It can be shown \cite{Hackbusch2010} Theorem 4.4.1, that for $f \in C(\tir{\Omega})$ the problem
\begin{align*} 
\begin{split}
\Delta_h [z_h] & = r_h(f) \, \text{in} \,  \Omega_h \\
z_h & = 0 \, \text{on} \, \partial \Omega_h,
\end{split}
\end{align*}
has a unique solution. We denote by $\Delta_h^{-1}$ the inverse of the operator $\Delta_h$ with homogeneous boundary conditions. Let $||\Delta_h^{-1}||$
denote the operator norm of $\Delta_h^{-1}$, i.e.
$$
||\Delta_h^{-1}||= \sup_{|v_h|_{\infty} \neq 0} \frac{ | \Delta_h^{-1} v_h|_{\infty}} {|v_h|_{\infty}}.
$$
By Theorem 4.4.1 of \cite{Hackbusch2010}, $||\Delta_h^{-1}||$ is bounded independently of $h$. We note that Theorem 4.4.1 of \cite{Hackbusch2010} is proven for dimension $n=2$ but the proof extends immediately to arbitrary dimension.

The main result of this section is the following theorem
\begin{thm} \label{main} Let $\mathcal{M}_h$ denote a Lipschitz continuous finite difference scheme such that the mapping $T_1: \mathcal{M}(\Omega^h) \to \mathcal{M}(\Omega^h)$ defined by
$$
T_1[v_h] = v_h + \frac{1}{\mu} \mathcal{M}_h[v_h],
$$
is a strict contraction for $\mu \geq \mu_0 >0$. Then for some $\mu_1 >0$, the mapping $T_2: \mathcal{M}(\Omega^h) \to \mathcal{M}(\Omega^h)$ defined by
$$
T_2[v_h] = v_h - \frac{1}{\mu} \Delta_h^{-1} \mathcal{M}_h[v_h],
$$
is also a strict contraction for $\mu \geq \mu_1$. 
\end{thm}

\begin{proof}
By assumption, there exists a constant $C_1$ such that $0 <C_1<1$  and
\[
|T_1[v_h]-T_1[w_h]|_{\infty} \leq C_1|v_h-w_h|_{\infty},
\]
for all $v_h, w_h \in \mathcal{M}(\Omega^h)$. 
One may decompose $T_2[v_h]-T_2[w_h]$ as
\begin{align*}
T_2[v_h]-T_2[w_h]&=
 T_2[v_h]-T_1[v_h] +T_1[v_h]-T_1[w_h]+T_1[w_h]-T_2[w_h]
 \\
 &=(T_1[v_h]-T_1[w_h])+(T_2[v_h]-T_1[v_h] )-(T_2[w_h]-T_1[w_h]).
\end{align*}
Moreover
\[
T_1[v_h]-T_2[v_h]=\frac1\mu\left(\mathcal{M}_h[v_h]+\Delta^{-1}_d \mathcal{M}_h[v_h]\right)
=\frac1\mu\left(I+\Delta^{-1}_d \right)\mathcal{M}_h[v_h],
\]
where $I$ denotes the identity operator on $\mathcal{M}(\Omega^h)$.
We then get
\[
(T_2[v_h]-T_1[v_h] )-(T_2[w_h]-T_1[w_h])=-\frac1\mu\left(I+\Delta^{-1}_d \right)(\mathcal{M}_h[v_h]-\mathcal{M}_h[w_h]).
\]
We recall that $\mathcal{M}_h$ is Lipchitz continuous, i.e.
\[
|\mathcal{M}_h[v_h]-\mathcal{M}_h[w_h]|_{\infty}
\leq K |v_h-w_h|_{\infty},\, \forall v_h, w_h \in \mathcal{M}(\Omega^h).
\]
One deduces that
\begin{align*}
|T_2[v_h]-T_2[w_h]|_{\infty}&\leq
 |T_1[v_h]-T_1[w_h]|_{\infty} +|\frac1\mu\left(I+\Delta^{-1}_d \right)(\mathcal{M}_h[v_h]-\mathcal{M}_h[w_h])|_{\infty}
 \\
 &\leq C_1|v_h-w_h|_{\infty}+\frac K\mu  || I+\Delta^{-1}_d || \, |v_h-w_h|_{\infty}
  \\
 &\leq \left(C_1+\frac K\mu |I+\Delta^{-1}_d \|\right) |v_h-w_h|_{\infty}.
\end{align*}
Since  $\| I + \Delta^{-1}_d \| \leq \|I\| +  \| \Delta^{-1}_d  \|$ is  bounded independently of the discretization step $h$ and  $0 < C_1<1$, one may choose $\mu$ big enough such that
\[
C_1+\frac K \mu \| I+\Delta^{-1} \|<1,
\]
 making $T_2$ a strict contraction mapping. This concludes the proof. 
\end{proof}
Under the assumption of the above theorem, both the iterative methods \eqref{m1h-iter} and \eqref{m1h-iter2} converge linearly to the unique solution $u_h$ of \eqref{m1h}.

\subsection{A numerical example} \label{numerical}

We say that a mesh function $v_h$ is {\it discrete convex} if and only if $\Delta_e v_h(x) = v_h(x+e) -2 v_h(x) + v_h(x-e) \geq 0$ for all $x \in \Omega_h$ and  $e \in \mathbb{Z}^d_h$ for which $\Delta_e v_h(x)$ is defined. Then the uniform limit of
discrete convex mesh functions is convex \cite[Lemma 2.11]{DiscreteAlex}. Moreover a bounded sequence of such functions is locally equicontinuous \cite{DiscreteAlex}. 

Following \cite{Oberman2010a}, we define
\begin{equation*} 
M_h[v_h](x) = \inf_{(\alpha_1,\ldots,\alpha_n) \in W_h(x)} \prod_{i=1}^d \max \bigg( \frac{v_h(x+ \alpha_i) -2 v_h(x) + v_h (x-\alpha_i)}{|\alpha_i|^2},0 \bigg).
\end{equation*}
where for $x \in \Omega_h$, $W_h(x)$ denotes the set of orthogonal bases of $\R^d$ such that for $(\alpha_1,\ldots,\alpha_n) \in W_h(x)$, $x\pm \alpha_i \in \Omega_h$, for all $i$. 

It is known that $\mathcal{M}_h[v_h]$ satisfies the assumptions of degenerate ellipticity and Lipschitz continuity as defined by Oberman \cite{Oberman06}. The consistency of the scheme was proved in \cite{Oberman2010a} while for a proof of stability, we refer to \cite{DiscreteAlex}. Note that $\mathcal{M}_h[v_h] \geq 0$ implies that $v_h$ is discrete convex. Hence the discrete convexity assumption is enforced in the discretization. Moreover, as pointed out in \cite{DiscreteAlex}, if one considers $M_h[v_h](x) + \epsilon v_h(x)$ where $\epsilon$ is taken close to machine precision, the discretization is proper and hence uniqueness holds.

For the numerical experiments, the space dimension $d$ is taken as 2 and the computational domain is the unit square $(0,1)^2$. Numerical experiments with $\nu$ a Dirac mass was reported in earlier papers, e.g. \cite{Oberman2010a}. Here we consider the example of \cite{Benamou2014b} where $\nu$ is the sum of two Dirac masses, i.e. we take
\[ u(x,y) = \left\{ 
  \begin{array}{l l}
    |y-\frac{1}{2}| & \quad \text{if $\frac{1}{4} < x < \frac{3}{4}$}\\
    \min\bigg\{ \, \sqrt{(x-\frac{1}{4})^2 + (y-\frac{1}{2})^2}, \sqrt{(x-\frac{3}{4})^2 + (y-\frac{1}{2})^2}  \, \bigg\} & \quad \text{otherwise},
  \end{array} \right.\]
and $\nu = \pi/2 \, \delta_{(1/4,1/2)} +  \pi/2 \, \delta_{(3/4,1/2)}$.
For simplicity, we only use a 17 point stencil. The initial guess is taken as the exact solution and the nonlinear equations solved with \eqref{m1h-iter2}. Errors are given in the maximum norm and reported 
on Table \ref{table-1}. 

\begin{table}
\begin{tabular}{c|cccccc} 
 \multicolumn{7}{c}{$h$}\\
$ \mu$  &  $1/2^3$ &  $1/2^4$&  $1/2^5$ &  $1/2^6$ &  $1/2^7$ &  $1/2^8$\\ \hline
50  & 4.71 $10^{-1}$ & 2.86 $10^{-1}$ & 1.69 $10^{-1}$ & 9.77 $10^{-2}$ & 5.50 $10^{-2}$ & 3.02 $10^{-2}$ \\
\end{tabular}
\caption{ }\label{table-1}
\end{table}



\end{document}